\documentclass[leqno,a4paper,11pt]{article}
\usepackage{amsmath,amssymb,ascmac,latexsym,amsthm, enumerate,color}
\usepackage[dvips]{graphicx}
\theoremstyle{definition}
\newtheorem{Assu}{Assumption}

\newtheorem{Th}{Theorem}[section]
\newtheorem{Def}[Th]{Definition}
\newtheorem{rem}[Th]{Remark}
\newtheorem{Lem}[Th]{Lemma}
\newtheorem{Prop}[Th]{Proposition}

\title{Construction of modified wave operators via wave packet transform}
\author{Taisuke Yoneyama}
\date{\today}
\newcommand{\dint}{\displaystyle\int}

\newcommand{\vp}{\varphi}
\newcommand{\vpp}{\varphi_0,\psi_0}

\newcommand{\HH}{\mathcal{H}}
\newcommand{\slim}{\operatorname*{s-lim}}
\newcommand{\supp}{\operatorname{supp}}

\begin{document}
\maketitle
\begin{abstract}
We construct modified wave operators for Schr\"odinger equations with time-dependent long-range potentials by means of the wave packet transform.
In contrast to previous works, our approach allows the construction of modified wave operators under weaker regularity assumptions on the potential.
More precisely, assuming only limited smoothness of the long-range part, we establish the existence of the modified wave operators and provide a rigorous proof within the wave packet transform framework.
Our results extend earlier constructions by relaxing regularity conditions and illustrate the effectiveness of the wave packet transform in the analysis of long-range scattering phenomena.
\end{abstract}
%
%
\section{Introduction}
In this paper,
we consider the following Schr\"odinger equation with time-dependent long-range potentials:
\begin{align}
\label{1}
i\frac{\partial}{\partial t}u = H(t)u, \quad H(t)= H_0+V(t),\quad H_0
=-\frac12\sum_{j=1}^{n}\displaystyle\frac{\partial^2}{\partial x_j^2}=-\frac12\Delta
\end{align}
in the Hilbert space $\mathcal{H}=L^2(\mathbb{R}^n)$
and the domain $D(H_0)=H^2(\mathbb{R}^{n})$ is the Sobolev space of order two.
The perturbation $V(t)=V_L(t)+V_S(t)$ satisfies the following conditions:
\setcounter{Assu}{18}
\begin{Assu}\label{ass-s}
$V_S(t)$ satisfies the following conditions:

\begin{enumerate}[(i)]
\item
$V_S(t)$ is symmetric in $\HH$ and the domain $D(V_S(t))\supset H^2(\mathbb{R}^{n})$ for all $t\in\mathbb{R}$.\\

\item
$\displaystyle\sup_{t\in\mathbb{R}}\|V_S(t)\|_{H_0}<1$, that is,
there exist real constants $a\in[0,1)$ and $b\geq0$ such that
\begin{align*}
\sup_t\|V_S(t)u\|_{\HH}\leq a\|H_0u\|_{\HH}+b\|u\|_{\HH}
\end{align*}
for any $u\in H^2(\mathbb{R}^n)$.\\

\item
$V_S(t)$ is the multiplication operator of $V_S(t,x)$ which is a real-valued Lebesgue measurable function of $(t,x)\in\mathbb{R}\times\mathbb{R}^{n}$ and
\begin{align*}
\sup_t\|(1+|x|)^{1+\delta}V_S(t,x)\|_{L_x^\infty}\leq C
\end{align*}
for some positive constants $\delta, C$.
\item
$V_S(t)f$ is strongly differentiable in $\mathcal{H}$ for $f\in H^2(\mathbb{R}^n)$.
\end{enumerate}
\end{Assu}

\begin{rem}
The assertion (iii) can be substituted by the fact that
there exists a bounded and monotone function $h\in L^1(0,\infty)$ such that
\begin{align*}
\sup_t\|V_S(t)(H_0+1)^{-1}\chi_{\{|x|\geq R\}}\|_{\mathrm{op}}\leq h(R)
\end{align*}
for any $R\geq0$, where $\|\cdot\|_{\mathrm{op}}$ is the operator norm on $\HH$ and $\chi$ is a characteristic function, that is, $\chi_{\{|x|\geq R\}}=1$ if $|x|\geq R$ and $\chi_{\{|x|\geq R\}}=0$ if $|x|< R$.
\end{rem}

\setcounter{Assu}{11}
\begin{Assu}\label{ass-l}

$V_L(t)$ satisfies the following conditions:

\begin{enumerate}[(i)]
\item
$V_L(t)$ is the multiplication operator of $V_L(t,x)$ which is a real-valued Lebesgue measurable function of $(t,x)\in\mathbb{R}\times\mathbb{R}^{n}$.\\

\item $V_L(t,x)$ is in $C^4(\mathbb{R}\times\mathbb{R}^{n})$
and there exists a real constant $\delta\in(0,1]$ such that for any $\alpha\in\mathbb{Z}_+^n$,
\begin{align}
\label{VL}
|\partial_x^{\alpha}V_L(t,x)|\leq 
C_\alpha\langle x\rangle^{-\delta-|\alpha|}
\end{align}
for some positive constant $C_\alpha$ and for any $(t,x)\in\mathbb{R}\times\mathbb{R}^{n}$,
where $\langle\cdot\rangle=\sqrt{1+|\cdot|^2}$.\\

\end{enumerate}
\end{Assu}
Under Assumption \ref{ass-s},
there exists a family of unitary operators $(U(t,s))_{(t,s)\in\mathbb{R}^2}$ in $\mathcal{H}$  satisfying the following conditions:
\begin{enumerate}[(i)]
\item
 For $f\in \mathcal{H}$, $U(t,s)f$ is strongly continuous function with respect to $t$ and $s$ and satisfies
\[
U(t,t')U(t',s)f=U(t,s)f,\,U(t,t)f=f\quad\mbox{for all }t,t',s\in\mathbb{R}.
\]
\item
There exists a function space $\mathcal{D}$ with $\mathcal{S}(\mathbb{R}^n)\subset\mathcal{D}\subset H^2(\mathbb{R}^n)$
such that if $f\in \mathcal{D}$, $U(t,s)f$ is in $\mathcal{D}$, is strongly continuously differentiable in $\mathcal{H}$ with respect to $t$ and $s$ and satisfies
\[
\frac{\partial}{\partial t}U(t,s)f=-iH(t)U(t,s)f
\quad\mbox{and}\quad
\frac{\partial}{\partial s}U(t,s)f=iU(t,s)H(s)f
\]
for all $t,s\in\mathbb{R}$, where $\mathcal{S}(\mathbb{R}^n)$ is the Schwartz space of all rapidly decreasing functions on $\mathbb{R}^n$.
\end{enumerate}

We construct modified wave operators for Schr\"odinger equations with time-dependent long-range potentials by using wave packet transform
which is defined by A. C\'ordoba and C. Fefferman \cite{CF}.
We also give a proof of the existence of the modified wave operators.
\begin{Def}[wave packet transform]\label{wpt-def}
Let $\vp,\psi\in\mathcal{S}(\mathbb{R}^n)\setminus\left\{ 0\right\}$ with $(\vp,\psi)_{\mathcal{H}}\neq0$,
$f$ be a tempered distribution on $\mathbb{R}^{n}$ and $F(y,\xi)$ be a function on $\mathbb{R}^{n}\times\mathbb{R}^{n}$.
We define the wave packet transform $W_\varphi f(x,\xi)$ of $f$ 
with the wave packet generated by a function $\varphi$ as follows:
\[
W_\varphi f(x,\xi)=\int_{\mathbb{R}^{n}}\overline{\varphi(y-x)}f(y)e^{-iy\xi}dy
\quad\mbox{for } (x,\xi)\in\mathbb{R}^{n}\times\mathbb{R}^{n}
\]
and its inverse $W_{\psi}^{-1}$ is defined by
\[
W_{\psi,\vp}^{-1} F(x)=\frac{1}{(2\pi)^n(\psi,\vp)_{\mathcal{H}}}\int\int_{\mathbb{R}^{2n}}\psi(x-y)F(y,\xi)e^{ix\xi}dyd\xi
\quad\mbox{for }x\in\mathbb{R}^{n}.
\]
\end{Def}

\begin{rem}
The above inverse is ``left'' inverse.
Thus 
\begin{align}
\label{invpro}
W_{\psi,\vp}^{-1}[W_\varphi f]=f
\end{align}
holds, but $W_\varphi [W_{\psi,\vp}^{-1}F]=F$ does not hold. 
\end{rem}
In this paper, we call $\varphi$, $\psi$ windows.

\begin{Def}[modified propagator]\label{mwo-def}
Let $f\in \HH$, $\vp_0,\psi_0\in\mathcal{S}(\mathbb{R}^n)\setminus\left\{ 0\right\}$ with $(\vp_0,\psi_0)_{\mathcal{H}}\neq0$.
We define the modified propagator $U_M(t,0)$ as
\begin{align*}
U_M^{\vpp}(t,0)f\equiv W^{-1}_{\psi(t),\vp(t)}\left[e^{-i\int_{0}^tp(s,x(s;t),\xi(s;t))\;ds}W_{\varphi_0}f(x(0;t),{\xi})\right],
\end{align*}
where $\vp(t)=e^{-itH_0}\vp_0$, $\psi(t)=e^{-itH_0}\psi_0$, $p(s,x,\xi)=|\xi|^2/2+V_L(s,x)-x\cdot\nabla_xV_L(s,x)$ and
$x(s;t)=x(s;t,x,\xi)$ and $\xi(s;t)=\xi(s;t,x,\xi)$ are solutions of the Newton second law 
under the potential $V_L(t,x)$ of motion with initial data at time $t$, that is,
\begin{align}
\label{moeq}
&\begin{cases}
\dfrac{d}{ds}x(s;t)=\xi(s),\quad x(t)=x,\\[10pt]
\dfrac{d}{ds}\xi(s;t)=-\nabla_xV_L(s,x(s;t)),\quad \xi(t)=\xi.
\end{cases}
\end{align}
\end{Def}
\begin{rem}\label{rem1}
If $V_L(t,x)=0$, we have $x(s;t)=x+(s-t)\xi$, $\xi(s;t)\equiv\xi$ and
\begin{align}
\label{siki-free}
U_M^{\vpp}(t,0)f
=W_{\psi(t),\vp(t)}^{-1}\left[e^{-it|\xi|^2/2}W_{\vp_0}f(x-t\xi,\xi)\right],
\end{align}
which is the representation of $e^{-itH_0}f$ introduced in \cite{KKS2} (proved in Section 2).
\end{rem}

The main result of this paper is the following.
\begin{Th}\label{main-theorem}
Suppose that \ref{ass-s} and \ref{ass-l} are satisfied.

Then for any $f \in\mathcal{F}^{-1}\Big[C_0^\infty(\mathbb{R}^n\setminus\{0\})\Big]$, 
the limits 
\begin{align}
\label{maim-mwo}
\lim_{t\to\pm\infty}U(t,0)^*U_M^{\vpp}(t,0)f
\end{align}
exist in $\HH$ for some $\vp_0,\psi_0\in\mathcal{S}(\mathbb{R}^n)\setminus\left\{ 0\right\}$ with $(\vp_0,\psi_0)_{\mathcal{H}}\neq0$,
where $\mathcal{F}$ is the Fourier transform defined by
$\mathcal{F}[u](\xi)=\hat{u}(\xi)=(2\pi)^{-n/2}\int_{\mathbb{R}^{n}} e^{-ix\cdot\xi}u(x)dx$.

\end{Th}
\begin{rem}
For each $f \in L^2(\mathbb{R}^n)$, the convergence in 
Theorem~\ref{main-theorem} holds in the norm topology of 
$L^2(\mathbb{R}^n)$ and is therefore strong rather than weak.
Since the modifier depends on the choice of the windows 
$\vp_0$ and $\psi_0$, the convergence is not formulated 
as a single strong operator limit independent of these 
auxiliary functions, but appears in the above form.
\end{rem}
It should be noted that the windows $\vp_0,\psi_0$  depend on $f$.
Cutting low energy leads to the following theorem.
\begin{Th}\label{sub-theorem}
Suppose that \ref{ass-s} and \ref{ass-l} are satisfied.

Then for any positive number $d$, there exists $\vp_0,\psi_0\in\mathcal{S}(\mathbb{R}^n)\setminus\left\{ 0\right\}$ with $(\vp_0,\psi_0)_{\mathcal{H}}\neq0$
such that the modified wave operators 
\begin{align}
\label{sub-mwo}
W_M^{\pm,\vpp}\equiv\slim_{t\to\pm\infty} U(0,t)U_M(t,0)P_d
\end{align}
exist,
where $P_d=\mathcal{F}^{-1}\chi_{\{|\xi|\geq d\}}\mathcal{F}$.
\end{Th}
Modified wave operators have long stood at the center of long-range scattering theory. 
Foundational developments can be found in the classical works of H\"ormander \cite{Ho} and Derezi\'nski--G\'erard \cite{DG}, 
where scattering for long-range Schr\"odinger operators is systematically treated within the framework of Fourier integral operators and microlocal analysis. 
Among the seminal contributions, the construction by Isozaki--Kitada \cite{IK} has become a cornerstone for Schr\"odinger operators with time-independent long-range potentials, 
relying on asymptotic solutions to the Hamilton--Jacobi equation tailored to the underlying classical dynamics.

In subsequent developments, scattering under reduced regularity assumptions has been investigated extensively. 
For instance, Ito--Skibsted \cite{IS} established modified scattering for long-range potentials under $C^2$-regularity assumptions, 
highlighting that high regularity is not essential for the construction of modified wave operators in the stationary setting. 
Time-dependent short-range problems have also been analyzed in detail; see, for example, Soffer--Wu \cite{SW}. 
These works collectively demonstrate the robustness of scattering theory beyond the classical smooth framework.

While the Hamilton--Jacobi-based construction remains remarkably effective in the stationary case, 
its direct adaptation to time-dependent long-range potentials is less straightforward. 
A complementary viewpoint, originating with H\"ormander and later extended to genuinely time-dependent systems by Kitada--Yajima \cite{KY}, 
constructs modifiers by incorporating the classical equations of motion associated with the evolving Hamiltonian.

In this paper, we refine this line of thought by combining the Hörmander-type modification with the wave packet transform.
Departing from the route taken by Kitada--Yajima, we utilize classical trajectories directly within a microlocal framework rather than constructing modifiers solely through global solutions of the Hamilton--Jacobi equation.
This approach yields a transparent and dynamically faithful representation of the modified evolution in phase space.
Under appropriate long-range decay assumptions, we establish the existence of modified wave operators for time-dependent Schrödinger equations assuming $C^{4}$-regularity in the spatial variable.
Crucially, our framework accommodates genuinely time-dependent long-range potentials and treats the temporal variation of the Hamiltonian in an intrinsic manner, without reducing the problem to a stationary approximation.
This provides a systematic and conceptually unified extension of modified scattering theory to the fully time-dependent long-range regime.
Our method, based on phase-space localization via the wave packet transform, offers a robust and structurally transparent approach to time-dependent long-range scattering phenomena.

We use the following notations throughout the paper.
For a subset $\Omega$ in $\mathbb{R}^{n}$ or in $\mathbb{R}^{2n}$,
we use the usual inner product $(f,g)_{L^2(\Omega)}=\int_\Omega f\bar{g}dx$
and norm $\|f\|_{L^2(\Omega)}=(f,f)_{L^2(\Omega)}^{1/2}$ for $f,g\in L^2(\Omega)$.
We write $\partial_{x_j}=\partial/\partial x_j$, $\partial_t=\partial/\partial t$,
$L^2_{x,\xi}=L^2(\mathbb{R}^{n}_x\times\mathbb{R}^{n}_\xi)$,
$(\cdot,\cdot)=(\cdot,\cdot)_{L^2_{x,\xi}}$, $\|\cdot\|=\|\cdot\|_{L^2_{x,\xi}}$,
$\|f\|_{\Sigma(l)}=\sum_{|\alpha+\beta|=l}\|x^\beta\partial^\alpha_x f\|_{\mathcal{H}}$
and $W_\varphi u(t,x,\xi)=W_\varphi[u(t)](x,\xi)$.
$\mathcal{F}^{-1}[f](x)=(2\pi)^{-n/2}\int_{\mathbb{R}^{n}} e^{ix\cdot\xi}f(\xi)d\xi$.
We often write $\{\xi=0\}$ to denote $\{(x,\xi)\in\mathbb{R}^{2n}|\,\xi=0\}$.

The plan of this paper is as follows.
In section 2, we recall the properties of the wave packet transform.
In section 3, we study the properties of the classical trajectories.
In section 4, we construct the modified propagators and give a proof of Theorem \ref{main-theorem} and Theorem \ref{sub-theorem}.
%
%
\section{Wave packet transform}
In this section, we recall the properties of the wave packet transform
and give $\eqref{siki-free}$ and the representation of solutions to $\eqref{1}$ via wave packet transform,
which is introduced in \cite{KKS2}.
\begin{Prop}\label{wpt-prop2}
Let $\varphi, \psi\in \mathcal{S}\setminus\{ 0\}$ with $(\vp,\psi)_{\mathcal{H}}\neq0$, $f,g\in \mathcal{S}'$ and $F\in L^2_{x,\xi}$.\\
Then we have $(W_{\vp}f,W_{\psi}g)=(\psi,\vp)_{\mathcal{H}}(f,g)_{\mathcal{H}}$,
\begin{align}
\label{siki-wpt-Four}
W_{\vp}f(x,\xi)
&=e^{-ix\xi}\int_{\mathbb{R}^{n}}\overline{\hat{\varphi}(\eta-\xi)}\hat{f}(\eta)e^{ix\eta}\;d\eta
\end{align}
\begin{align}
\label{siki-wpt-inv}
(W^{-1}_{\psi,\vp} F,g)_{\HH}
&=\frac{1}{(2\pi)^n(\psi,\vp)_{\mathcal{H}}}(F,W_{\psi}g)
\end{align}
and
\begin{align}
\label{inv-bdd}
\|W^{-1}_{\psi,\vp}F\|_\HH\leq\dfrac{\|\psi\|_\HH}{(2\pi)^n|(\vp,\psi)_{\mathcal{H}}|}\|F\|.
\end{align}
\end{Prop}
\begin{Prop}\label{wpt-prop}
Let $\varphi=\vp(t,x)\in C^1(\mathbb{R};\mathcal{S}\setminus\{ 0\})$ and $f=f(t,x)\in C^1(\mathbb{R};\mathcal{S}')$.\\
Then we have
$W_{\vp(t)} f\in C^1(\mathbb{R};C^\infty(\mathbb{R}^{n}\times\mathbb{R}^{n}))$,
\begin{align}
\label{wpt1}
W_{\varphi(t)}[\partial_xf](t,x,\xi)
=-i\xi W_{\varphi(t)}f(t,x,\xi)+W_{\partial_x\varphi(t)}f(t,x,\xi)
\end{align}
and
\begin{align}
\label{wpt2}
W_{\varphi(t)}[\partial_tf](t,x,\xi)
=\partial_tW_{\varphi(t)}f(t,x,\xi)-W_{\partial_t\varphi(t)}f(t,x,\xi)
\end{align}
for any $(t,x,\xi)\in\mathbb{R}\times\mathbb{R}^{n}\times\mathbb{R}^{n}$.
\end{Prop}
The proofs of the above statements are obtained by the Plancherel theorem and elementary calculations and are omitted here.

Next, we transform $\eqref{1}$ by using the wave packet transform with the time-dependent window.
Let $\vp_0\in\mathcal{S}\setminus\{ 0\}$, $u_0\in \HH$ and $\vp(t)=e^{-itH_0}\vp_0$.
We put
\begin{align*}
R_S(t,x,\xi;f)
&=W_{\vp(t)}[V_S(t)U(t,0)f](x,\xi),\\
R^{\alpha}_L(t,x,\xi;f)
&=\int_{\mathbb{R}^n}
\overline{\vp(t,y-x)}(y-x)^{\alpha}V_L^{(\alpha)}(t,x,y)U(t,0)f(y)e^{-iy\xi}dy
\end{align*}
and
\begin{align*}
R(t,x,\xi;f)
=R_S(t,x,\xi;f)+\sum_{|\alpha|=2}R^{\alpha}_L(t,x,\xi;f)
\end{align*}
where $V_L^{(\alpha)}(t,x,y)=\dint_0^1(\partial_x^{\alpha}V_L)(t,x+\theta(y-x))(1-\theta)d\theta $.
By $\eqref{wpt1}$, $\eqref{wpt2}$ and the Taylor expansion 
$$V_L(t,y)=V_L(t,x)+(y-x)\cdot\nabla_xV_L(t,x)+\sum_{|\alpha|=2}(y-x)^{\alpha}V_L^{(\alpha)}(t,x,y),$$
$(\ref{1})$ with initial data $u(0)=f$ is transformed to
\begin{align*}
\left(i\partial_t+i\xi\cdot\nabla_x-i\nabla_xV_L(t,x)\cdot\nabla_\xi-\tilde{V}_L(t,x)-\frac{1}{2}|\xi|^2 \right)W_{\vp(t)}u(t,x,\xi)
= R(t,x,\xi,f),
\end{align*}
where $\tilde{V}_L(t,x)=V_L(t,x)-x\cdot\nabla_xV_L(t,x)$.
The method of characteristics implies that
\begin{align*}
W_{\varphi(t)}[U(t,0)\psi](x,\xi)=&
e^{-i\int_{0}^t \left(\frac{1}{2}|\xi(s)|^2+\tilde{V}_L(s,x(s;t))\right)ds}W_{\vp}\psi(x(0;t),\xi(0;t))\\
&-i\int_0^te^{-i\int_{s}^t \left(\frac{1}{2}|\xi(s_1)|^2+\tilde{V}_L(s_1,x(s_1))\right)ds_1}R(s,x(s;t),\xi(s),f)ds,
\end{align*}
where $x(s;t)=x(s;t,x,\xi)$ and $\xi(s;t)=\xi(s;t,x,\xi)$ are the solutions to $\eqref{moeq}$.
Thus we obtain the representation of the solution to $\eqref{1}$ by the wave packet transform.
In particular, if $V(t,x)\equiv0$, we get $\eqref{siki-free}$.

\begin{Lem}\label{wpt-dense}
Let $\varphi_0\in \mathcal{S}\setminus\{ 0\}$.
Then the space
\begin{align*}
\left\{ f\in \mathcal{S}\;\middle|\;W_{\vp_0}f\in C_0^\infty(\mathbb{R}^{2n}\setminus\{\xi=0\})\right\}
\end{align*}
is dense in $\mathcal{H}$.
\end{Lem}
Let $f\in\HH$ and $\varepsilon$ be a fixed positive number.
Since $C_0^\infty(\mathbb{R}^{2n}\setminus\{\xi=0\})$ is dense in $L^2(\mathbb{R}^{2n})$, there exists $\omega\in C_0^\infty(\mathbb{R}^{2n}\setminus\{\xi=0\})$ satisfying
$
\|W_{\vp_0} f-\omega\|\leq\varepsilon.
$
Putting $f_0\equiv W_{\vp_0,\vp_0}^{-1}\omega$, we have $\|f-f_0\|_{\HH}\leq(2\pi)^{-n}\|\vp_0\|^{-1}\varepsilon$ by $\eqref{wpt2}$.

%
%
\section{Classical trajectories}
In this section, we consider the classical orbits affected by the long-range potential
and construct the solution of the motion equation.
For simplicity, we write $V$ as $V_L$ in this section.
\begin{Prop}\label{orbit}
Suppose that \ref{ass-l} is satisfied.
Then the solution of $\eqref{moeq}$ exists uniquely
and is in $C^{3}(\mathbb{R}^{2n})$ for $s\in\mathbb{R}$.
\end{Prop}
\begin{proof}
The mappings $(x,\xi)\mapsto (x(s;t),\xi(s;t))$ are diffeomorphisms,
which completes the proof.
\end{proof}
In addition to $\eqref{moeq}$, we define
$y(s)=y(s;t,x,\xi)$ and $\eta(s)=\eta(s;t,x,\xi)$ are the solutions to
\begin{align}
\label{ym}
&\begin{cases}
\dfrac{d}{ds}y(s)=\eta(s),\quad y(0)=x,\\[3mm]
\dfrac{d}{ds}\eta(s)=-\nabla_xV(s,y(s)),\quad \eta(t)=\xi
\end{cases}
\end{align}
For positive numbers $a,R$, we put
$\Gamma_{a,R}\equiv\{(x,\xi)\in\mathbb{R}^{2n}|a<|\xi|<R,\,|x|<R\}$.
\begin{Prop}\label{modi-orbit}
Suppose that \ref{ass-l} is satisfied.
Let $a,R$ be positive numbers.
Then there exists positive number $T$ such that the solution of $\eqref{ym}$ exists uniquely
and is in $C^{2}(\Gamma_{a,R})$ for $0<\pm s< \pm t$ and $\pm t>T$
and the following identities hold:
\begin{enumerate}[(i)]
\item
\begin{align}
\label{tra}
\begin{cases}
y(s;t,x(0;t,x,\xi),\xi)=x(s;t,x,\xi)\\
\eta(s;t,x(0;t,x,\xi),\xi)=\xi(s;t,x,\xi).
\end{cases}
\end{align}
\item
For any $\alpha, \beta\in\mathbb{Z}^n_+$ with $|\alpha+\beta|\leq2$, we have for $(x,\xi)\in\Gamma_{a,R}$
\begin{align}
\label{yb}
\langle y(s;t,x,\xi)\rangle
&\geq c\langle s\rangle,\\
\label{yu}
\langle \partial_x^\alpha\partial_\xi^\beta \{ y(s;t,x,\xi) -(x+s\xi)\}\rangle
&\leq C_{\alpha,\beta}\langle s\rangle^{-\delta-|\alpha|},\\
\label{eu}
\langle \partial_x^\alpha\partial_\xi^\beta \{ \eta(s;t,x,\xi) -\xi\}\rangle
&\leq C_{\alpha,\beta}\langle s\rangle^{-\delta-|\alpha|},
\end{align}
where $c,C_{\alpha,\beta}$ are positive and independent of $s,t,x$ and $\xi$.
\end{enumerate}
\end{Prop}
\begin{proof}
We shall give a proof only for the case $t>s>0$. The other case is proven similarly.
Let $t\in(T,\infty)$ be fixed.
($T>0$ is decided later.) 
$\eqref{ym}$ is equivalent to the following integral equation:
\begin{align}
\label{ymi}
y(s,t,x,\xi)&=x+\displaystyle\int_0^s\left(\xi-\int_t^{\sigma}\nabla_xV(\tau,y(\tau))\;d\tau\right)\;d\sigma\\
&=x+s\xi+\int_0^s\left(\int_0^\tau\nabla_xV(\tau,y(\tau))\;d\sigma\right)\;d\tau+\int_s^t\left(\int_0^s\nabla_xV(\tau,y(\tau))\;d\sigma\right)\;d\tau\nonumber\\
&=x+s\xi+\int_0^t\min\{\tau,s\}\nabla_xV(\tau,y(\tau))\;d\tau\nonumber
\end{align}
Hence by the Picard iteration, it suffices to prove 
\begin{align}
\label{pN}
\langle p_N(s,t;x,\xi)\rangle\geq c\langle s\rangle
\end{align}
for any $N\in\mathbb{N}$, $s\in(0,t)$, and $(x,\xi)\in\Gamma_{a,R}$,
where $\{p_N\}$ satisfies
\begin{align}
\label{pi}
&\begin{cases}
p_0(s,t;x,\xi)&=x+s\xi,\\
p_N(s,t;x,\xi)&=x+s\xi+\displaystyle\int_0^t\min\{\tau,s\}\nabla_xV(\tau,p_{N-1}(\tau,t;x,\xi))\;d\tau\quad (N\geq1).
\end{cases}
\end{align}
There exists $c_0>0$ independent of $(x,\xi)\in\Gamma_{a,R}$ satisfying $c_0\langle s\rangle<\langle x+s\xi\rangle<c_1\langle s\rangle$.
Thus we take $T>0$ such that $C_1\langle T\rangle^{-\delta}<c_0$, where $C_1=\max_{|\alpha|=1}C_{\alpha}$ and $C_{\alpha}$ is in $\eqref{VL}$.
Then $\eqref{pN}$ is obtained by induction.
The relation $\eqref{tra}$ is obtained by the uniqueness of solutions and the facts that $y(0;t,x(0;t,x,\xi),\xi)=x(0;t,x,\xi)$, $\eta(t;t,x(0;t,x,\xi),\xi)=\xi$.
\end{proof}

%
%
\section{Existence of modified wave operators}
In this section, we prove Theorem $\ref{main-theorem}$ and Theorem $\ref{sub-theorem}$ by the duality argument
and we find suitable windows $\vp_0$ and $\psi_0$.
In the following, we shall discuss the case $t>0$. The other case is discussed similarly.

First of all, we treat the short-range term.
Although the following lemmata can be proven in a similar way to \cite{YK}, we
give an outline of proofs for the readers' convenience.
\begin{Lem}\label{outgoing}
Let $a$ be a positive number and $\vp_0\in \mathcal{S}$ satisfying supp $\hat{\vp_0}\subset\{\xi\in\mathbb{R}^n\;|\;|\xi|<a\}$.
Then for any $N\geq0$ and $a'>a$, there exists a positive constant $C_N$ satisfying
\begin{align*}
|\chi_{\{|x|>a'|t|\}}e^{-itH_0}\vp_0(x)|
\leq C_N\langle t\rangle^{-N}\langle x\rangle^{-N}
\end{align*}
for any $(t,x)\in\mathbb{R}^{1+n}$.
\end{Lem}
The proof is given by integration by parts and elementary calculations.
\begin{Lem}\label{Rs}
Let  $a,R$ be positive numbers and suppose that \ref{ass-s} is satisfied.
Then we have
\begin{align}
\label{Spro}
\left|W_{\vp(s)}\left[V_S(s)g\right](y(s;t),\eta(s;t))\right|
\leq C\langle s\rangle^{-1-\delta}\|g\|_\HH
\end{align}
when $|x|>cs$, $t\gg1$, $s\in[0,t]$ and $g\in \mathcal{H}$.
\end{Lem}
\begin{proof}
We divide
\begin{align*}
W_{\vp(s)}&\left[V_S(s)g\right](x,\xi)\\
&=\int_{\mathbb{R}^n}\overline{\vp(t,y-x)}\chi_{\{|y-x|\leq cs/2\}}(y)V_S(s)g(y)e^{-iy\xi}dy\\
&+\int_{\mathbb{R}^n}\overline{\vp(t,y-x)}\chi_{\{|y-x|> cs/2\}}(y)V_S(s)g(y)e^{-iy\xi}dy\\
&=I_1+I_2
\end{align*}
By Lemma \ref{outgoing}, we have $|I_2|\leq C_N\langle s\rangle^{-N}\|g\|_\HH$.
The other part is proven by the fact that $|y|\geq|x|-|y-x|\geq cs/2$ and $|V_S(s,y)|\leq C(1+|y|)^{-1-\delta}$.
\end{proof}

For the long-range part, the right hand side of the identity
\begin{align}
\label{s2term}
x^\alpha e^{-isH_0}\vp_0= e^{-isH_0}(x+is\nabla)^\alpha \vp_0
\end{align}
seems to increase of order $s^2$ as $s \to \infty$ for $|\alpha|=2$.
Hence, instead of using the above expression as it stands, we control the long-range part through the following estimate.
\begin{Prop}\label{RP}
Let $a,R$ be positive numbers and suppose that \ref{ass-l} is satisfied.
For any multi-index $\alpha=(\alpha_1,\alpha_2)$ with $|\alpha|=2$, one has
\begin{align}
\label{rl-es}
&\Big|2e^{i\int_{0}^s \left(\frac{1}{2}|\eta(s_1;t)|^2+\tilde{V}_L(s_1,y(s_1;t))\right)ds_1}R^{\alpha}_L(s,y(s;t),\eta(s;t),g)\\\notag
&+\partial^\alpha_{\xi}\left(e^{i\int_{0}^s \left(\frac{1}{2}|\eta(s_1;t)|^2+\tilde{V}_L(s_1,y(s_1;t))\right)ds_1}\rho_{\alpha}(s,y(s;t),\eta(s;t),g)\right)\Big|
\leq C \langle s\rangle^{-1-\delta}\|g\|_\HH
\end{align}
for any $(x,\xi)\in\Gamma_{a,R}$, $0<s<t$ and $g\in \mathcal{H}$,
where 
\begin{align*}
\rho_{\alpha}(s,x,\xi,g)=\int_{\mathbb{R}^n}\overline{\vp(y-x)}
\left(\dint_0^1V^{(\alpha)}_L(s,x+\theta(y-x))(1-\theta)d\theta\right) g(y)e^{-iy\xi}\;dy
\end{align*}
and $C$ is independent of $x$, $\xi$, and $g$.
\end{Prop}
\begin{proof}

By $\eqref{yb}$, $\eqref{yu}$ and $\eqref{eu}$,  it suffices to prove
\begin{align}
\label{its-es}
\left|2e^{is|\xi|^2/2}R^{\alpha}_L(s,x+s\xi,\xi,g)
+\partial^\alpha_{\xi}\left(e^{is|\xi|^2/2}\rho_{\alpha}(s,x+s\xi,\xi,g)\right)\right|\leq C \langle s\rangle^{-1-\delta}\|g\|_\HH.
\end{align}

In this proof, though $y$ is the variable in integration,  we use the notation for $\Phi\in C^{\infty}(\mathbb{R}^{2n};\mathbb{R})$ and $k\in\mathbb{N}$
\begin{align*}
r(x,&\xi,\vp,\Phi(x,y),V,k)\\
&\equiv\int_{\mathbb{R}^n}\overline{\vp(y-x)}
\Phi(x,y)\left(\dint_0^1V(s,x+\theta(y-x))(1-\theta)^{k}d\theta\right) g(y)e^{-iy\xi}\;dy
\end{align*}
which, together with $\eqref{s2term}$ leads
\begin{align}
R^{\alpha}_L(s,x,\xi,g)
=&r(x,\xi,\vp(s),(y-x)^\alpha,V^{(\alpha)}_L,1)\\
=&r(x,\xi,\vp_{(\alpha)}(s),1,V^{(\alpha)}_L,1)+2is\sum_{|\beta|=1}r(x,\xi,\vp_{(\beta)}^{(\alpha-\beta)}(s),1,V^{(\alpha)}_L,1)\notag\\
&-s^2r(x,\xi,\vp^{(\alpha)}(s),1,V^{(\alpha)}_L,1)\notag
\end{align}
{and}
\begin{align*}
\rho_{\alpha}(s,x+s\xi,\xi,g)=&r(x,\xi,\vp(s),1,V^{(\alpha)}_L,1)
\end{align*}
Here we denote $\vp_{(\beta)}^{(\alpha)}(s)=e^{-isH_0}(x^\beta\partial^{\alpha}\vp_{0})$.

Hence, we have
\begin{align}
\partial^\alpha_{\xi}&\left(e^{is|\xi|^2/2}\rho_{\alpha}(s,x+s\xi,\xi,g)\right)\\
&=-s^2\xi^\alpha e^{is|\xi|^2/2}r(x+s\xi,\xi,\vp(s),1,V^{(\alpha)}_L,1)+i\delta_{\alpha_1,\alpha_2}s^2e^{is|\xi|^2/2}r(x+s\xi,\xi,\partial^\alpha\vp(s),1,V^{(\alpha)}_L,1)\notag\\
&+s^2e^{is|\xi|^2/2}r(x+s\xi,\xi,\vp^{(\alpha)}(s),1,V^{(\alpha)}_L,1)
+s^2e^{is|\xi|^2/2}r(x+s\xi,\xi,\vp(s),1,\partial^\alpha V^{(\alpha)}_L,2)\notag\\
&-e^{is|\xi|^2/2}r(x+s\xi,\xi,\vp(s),y^\alpha,V^{(\alpha)}_L,1)
+\mbox{(1st-order term)}\notag\\
&=-e^{is|\xi|^2/2}r(x+s\xi,\xi,\vp(s),(y-x-s\xi)^\alpha,V^{(\alpha)}_L,1)+x^\alpha e^{is|\xi|^2/2}r(x+s\xi,\xi,\vp(s),1,V^{(\alpha)}_L,1)\notag\\
&+i\delta_{\alpha_1,\alpha_2}s^2e^{is|\xi|^2/2}r(x+s\xi,\xi,\partial^\alpha\vp(s),1,V^{(\alpha)}_L,1)+s^2e^{is|\xi|^2/2}r(x+s\xi,\xi,\vp(s),1,\partial^\alpha V^{(\alpha)}_L,2)\notag\\
&+s^2e^{is|\xi|^2/2}r(x+s\xi,\xi,\vp^{(\alpha)}(s),1,V^{(\alpha)}_L,1)+\mbox{(1st-order term)}\notag\\
&=-e^{is|\xi|^2/2}r(x+s\xi,\xi,\vp_{(\alpha)}(s),1,V^{(\alpha)}_L,1)-2ise^{is|\xi|^2/2}\sum_{|\beta|=1}r(x,\xi,\vp_{(\beta)}^{(\alpha-\beta)}(s),1,V^{(\alpha)}_L,1)\notag\\
&+2s^2e^{is|\xi|^2/2}r(x+s\xi,\xi,\vp^{(\alpha)}(s),1,V^{(\alpha)}_L,1)+x^\alpha e^{is|\xi|^2/2}r(x+s\xi,\xi,\vp(s),1,V^{(\alpha)}_L,1)\notag\\
&+i\delta_{\alpha_1,\alpha_2}s^2e^{is|\xi|^2/2}r(x+s\xi,\xi,\partial^\alpha\vp(s),1,V^{(\alpha)}_L,1)+s^2e^{is|\xi|^2/2}r(x+s\xi,\xi,\vp(s),1,\partial^\alpha V^{(\alpha)}_L,2)\notag\\
&+\mbox{(1st-order term)}.\notag
\end{align}
The simple calculation and $\eqref{yb}$ show that 
\begin{align*}
|r(x+s\xi&,\xi,\vp_{(\alpha)}(s),1,V^{(\alpha)}_L,1)|\\
&=\left|\int_{\mathbb{R}^n}\overline{\vp(y-x-s\xi)}\left(\dint_0^1V^{(\alpha)}_L(s,x+s\xi+\theta(y-x-s\xi))(1-\theta)d\theta\right) g(y)e^{-iy\xi}\;dy\right|\\
&\leq C \langle s\rangle^{-2-\delta}\int_{\mathbb{R}^n}\left|\overline{\vp(y-x-s\xi)}\right|\left|g(y)\right|\;dy
\leq C \langle s\rangle^{-2-\delta}\|g\|_\HH
\end{align*}
for $(x,\xi)\in\Gamma_{a,R}$ and $s>0$.

Similarly, we get
\begin{align*}
|r(x+s\xi,\xi,\vp_{(\alpha)}(s),1,V^{(\alpha)}_L,1)|+|r(x+s\xi,\xi,\vp_{(\beta)}^{(\alpha-\beta)}(s),1,V^{(\alpha)}_L,1)|&\leq C \langle s\rangle^{-2-\delta}\|g\|_\HH,\\
|e^{is|\xi|^2/2}r(x+s\xi,\xi,\vp(s),1,\partial^\alpha V^{(\alpha)}_L,2)|&\leq C \langle s\rangle^{-4-\delta}\|g\|_\HH,\\
|x^\alpha e^{is|\xi|^2/2}r(x+s\xi,\xi,\vp(s),1,V^{(\alpha)}_L,1)|\leq R^2|e^{is|\xi|^2/2}r(x+s\xi,\xi,\vp(s),1,V^{(\alpha)}_L,1)|&\leq C \langle s\rangle^{-2-\delta}\|g\|_\HH,\\
|\mbox{(1st-order term)}|&\leq C \langle s\rangle^{-1-\delta}\|g\|_\HH,
\end{align*}
which implies $\eqref{its-es}$.

\end{proof}
\begin{proof}[Proof of Theorem \ref{main-theorem}]

We prove for the case $t\rightarrow+\infty$ only.
The other case can be proven similarly.

We fix $\vp_0\in\mathcal{S}(\mathbb{R}^n)$ with $\hat{\vp_0}(0)\not=0$ and
define $\mathcal{D}_0\equiv C_0^\infty(\mathbb{R}_{x,\xi}^{2n}\setminus\{\xi=0\})$.
Since $\mathcal{D}_0$ is dense in $L^2(\mathbb{R}_{x,\xi}^{2n})$
and $\eqref{inv-bdd}$,
the set in $\HH$ of all the functions $f$ satisfying
$W_{\vp_0} f\in \mathcal{D}_0$
is dense in $\HH$.
Hence we show the existence of $W_M^+u_0$ for $u_0\in \mathcal{D}_0$.
We fix $a,R>0$ satisfying 
\begin{align}
\label{spt}
\supp W_{\vp_0} u_0\subset\Gamma_{a,R}
\end{align}
and $\psi_0\in\mathcal{S}(\mathbb{R}^n)\setminus\{0\}$ with supp $\hat{\psi_0}\subset\{|\xi|<c/2\}$,
where $c$ is in Proposition \ref{modi-orbit}.
Then we note that $\supp \partial_{x_1}(W_{\vp_0} u_0)\subset\supp W_\vp u_0\subset\Gamma_{a,R}$.
For $t\geq0$, we have by $\eqref{siki-wpt-inv}$, $\eqref{tra}$ and change of variables $x(0;t)\mapsto x$.
\begin{align}
\label{whole}
&\left(U(0,t)U_M(t,0)u_0,g\right)_\HH\\
&=\left(U_M(t,0)u_0,U(t,0)g\right)_\HH\nonumber\\
&=\left(e^{-i\int_{0}^t \left(\frac{1}{2}|\xi(s)|^2+\tilde{V}_L(s,x(s))\right)ds}W_{\vp}u_0(x(0;t),\xi),W_{\vp}[U(t,0)g]\right)\nonumber\\
&=\left(W_{\vp}u_0(x(0;t),\xi),W_{\vp}g(x(0;t),\xi(0;t))\right)\nonumber\\
&+i\left(W_{\vp}u_0(x(0;t),\xi),\int_0^te^{i\int_{0}^s \left(\frac{1}{2}|\xi(s_1;t)|^2+\tilde{V}_L(s_1,x(s_1;t))\right)ds_1}R(s,x(s;t),\xi(s;t),g)ds\right)\nonumber\\
&=\left(W_{\vp}u_0(x,\xi),W_{\vp}g(x,\xi)\right)\nonumber\\
&+i\left(W_{\vp}u_0(x,\xi),\int_0^te^{i\int_{0}^s \left(\frac{1}{2}|\eta(s_1;t)|^2+\tilde{V}_L(s_1,y(s_1;t))\right)ds_1}R(s,y(s;t),\eta(s;t),g)ds\right)\nonumber
\end{align}
Taking the Cauchy sequence of $\eqref{whole}$, by Cook--Kuroda method, it suffices to prove
\begin{align}
\left|
\left(W_{\vp}u_0(x,\xi),\int_0^te^{i\int_{0}^s \left(\frac{1}{2}|\eta(s_1;t)|^2+\tilde{V}_L(s_1,y(s_1;t))\right)ds_1}R(s,y(s;t),\eta(s;t),g)ds\right)
\right|<\infty.
\end{align}

By $\eqref{spt}$, Lemma \ref{Rs} and Proposition \ref{RP}, we have 
\begin{align*}
&\left|\left(W_{\vp_0}u_0(x,\xi),-2\int_0^t\partial^\alpha_{\xi}\left(e^{i\int_{0}^s \left(\frac{1}{2}|\eta(s_1;t)|^2+\tilde{V}_L(s_1,y(s_1;t))\right)ds_1}\rho_{\alpha}(s,y(s;t),\eta(s;t),g)\right) +\|g\|_\HH\cdot O(s^{-1-\delta})ds\right)\right|\\
&=\left|\left(W_{\vp_0}[x^\alpha u_0](x,\xi),\int_0^te^{i\int_{0}^s \left(\frac{1}{2}|\eta(s_1;t)|^2+\tilde{V}_L(s_1,y(s_1;t))\right)ds_1}\rho_{\alpha}(s,y(s;t),\eta(s;t),g) +\|g\|_\HH\cdot O(s^{-1-\delta})ds\right)\right|\\
&\leq C\|g\|_\HH<\infty,
\end{align*}
where the constant $C$ does not depent on $g$.

Therefore we obtain the existence of the strong limit of $\eqref{maim-mwo}$, which completes the proof of Theorem \ref{main-theorem}.
\end{proof}

\begin{proof}[Proof of Theorem \ref{sub-theorem}]

We prove for the case $t\rightarrow+\infty$ only.
The other case can be proven similarly.

By the Plancherel theorem, the identity
\begin{align}
\label{siki-smow}
 W_{\vp_0} \left[ P_d u_0 \right](x,\xi)&=\left(\Gamma u_0,e^{i\cdot\xi}\vp_0(\cdot-x)\right)_\HH
 =e^{ix\xi}\left(\chi_{\{|\cdot|\geq d\}}\hat{u_0},e^{ix\cdot}\hat{\vp_0}(\cdot-\xi)\right)_\HH.
\end{align}
holds.
Since $d$ in $\eqref{siki-smow}$ does not depend on the choice of $u_0$, 
taking $g_0\in\mathcal{S}(\mathbb{R}^n)\setminus\{0\}$ with supp $\hat{g_0}\subset\{|\xi|<d/2\}$ we get the existence of $\eqref{sub-mwo}$ in a similar way to the proof of Theorem \ref{main-theorem}.

\end{proof}

\section*{Acknowledgements}
The author is grateful to Professor Keiichi Kato for useful discussions.

\section*{Declarations}
The authors declare that they have no known competing financial interests or personal relationships that could have influenced the work reported in this paper.
This work does not involve any data generation or analysis, as it is purely theoretical.


\end{document}